\documentclass[11pt,fleqn]{article}
\RequirePackage[OT1]{fontenc}
\RequirePackage{amsthm,amsmath}
\RequirePackage[round]{natbib}
\RequirePackage[colorlinks,citecolor=blue,urlcolor=blue]{hyperref}

\usepackage{amsmath,amsfonts,amssymb,euscript,mathrsfs}
\usepackage{slashbox}
\usepackage{float}
\usepackage[final]{pdfpages}
\usepackage{natbib}
\usepackage[textwidth=3cm, textsize=footnotesize]{todonotes} %%%%new command for comments

\usepackage{txfonts}
\usepackage[T1]{fontenc}
\usepackage{fullpage}
\usepackage{graphicx}
\usepackage{dsfont}
\makeatletter
\def\captionof#1#2{{\def\@captype{#1}#2}}
\makeatother
\def\1{\mbox{\bf 1}}
\def\R{\mathbb{R}}

\def\N{\mathbb{N}}
\def\P{\mathbb{P}}
\def\E{\mathbb{E}}

\def\R{\mathbb{R}}

\def\Z{\mathbb{Z}}

\newtheorem{theo}{Theorem}

\newtheorem{prop}{Proposition}

\newtheorem{cor}{Corollary}
\newtheorem{Def/Prop}{Definition-Proposition}
\newtheorem{Remark}{Remark}

\newcounter{exos}
\renewcommand\theexos{\arabic{exos}}

\newcounter{prob}
\renewcommand\theprob{\arabic{prob}}

\begin{document}
\author{Konstantinos Fokianos\footnote{Lancaster University, Department of Mathematics \& Statistics, Fylde College, Lancaster, LA1 4YF,
United Kingdom {\it Email: k.fokianos@lancaster.ac.uk}.} \and
Lionel Truquet \footnote{UMR 9194 CNRS CREST, ENSAI, Campus de Ker-Lann, rue Blaise Pascal, BP 37203, 35172 Bruz cedex, France. {\it Email: lionel.truquet@ensai.fr}.}
 }
\title{On Categorical Time Series Models with Covariates}
\date{First version: October  2017 \\ Second version: July 2018}
\maketitle

\begin{abstract}
\noindent
We study the problem of stationarity and ergodicity for autoregressive
multinomial logistic  time series models which possibly  include a latent process and are defined by
a GARCH-type recursive equation.  We improve considerably upon the existing conditions  about  stationarity and ergodicity
of those  models.
Proofs  are based on  theory  developed for
chains with complete connections.  A useful coupling technique is employed
for studying ergodicity of infinite order  finite-state stochastic processes which  generalize  finite-state Markov chains.
Furthermore, for the case of  finite order Markov chains, we  discuss ergodicity properties of a model which includes
strongly exogenous but not necessarily bounded  covariates.
\end{abstract}
\vspace*{1.0cm}

\footnoterule

\newpage

\section{Introduction}
The goal of this article is to improve upon  theoretical properties of regression based models for the analysis of categorical
time series that might include some covariates which are not  necessarily  bounded.
Binary time series are particular cases of a categorical time series and the results we obtain
apply to logistic autoregressive models.
The conditional distribution of a categorical time series given its past is multinomial which obviously  belongs to the multivariate exponential family of distributions. As such,
the theory of generalized linear models, see \citet{McCullaghandNelder(1989)}, can be applied for modeling different  types of categorical data; nominal, interval and scale. We  will be mostly concerned
with nominal data and therefore the multinomial logistic model is the natural candidate for model fitting; see
\citet{Fahrmeir2001} and \citet{Fokianos2002a}, among other references,  for further discussion on modeling issues regarding categorical data.
We emphasize that  finite state Markov chains provide a simple but prominent model  of a categorical time series where lagged values of the response affect the determination of  its future states.
However, Markov modeling in the context of categorical time series,  poses challenging problems. Indeed,
as the order of the Markov chain increases so does the number of free parameters; in fact, the number of free parameters increases  exponentially fast.  Furthermore,
the Markovian  property requires simultaneous  specification of the  dynamics of the response and any possible  covariates observed jointly; such a specification might not be possible, in general.

We will be studying models for binary and, more generally, categorical time series, which are of infinite order or they are
driven by a latent process or a feedback mechanism.
This type of models is quite analogous to GARCH models -see
\citet{Bollerslev(1986)}-
but they are defined in terms of conditional log--odds instead of conditional variances.
In particular,  feedback models  make possible low dimensional parametrization, yet they can accommodate quite complicated data structures.
Examples of feedback models, in the context of binary and categorical  time series have been studied recently by \citet{Fok1} and \citet{Fok2}, among others.
We will discuss their  results and we will compare them with our findings.
Models and inference for  binary time series  are research  topics that have been studied by several authors;  see
\cite{Kedem1980book}
for an early treatment. Regression modeling, in this context, has been
studied by \citet{Cox1981},
\cite{Stern1984},
and
\cite{SludKedem1994}, among others; see also \citet[Ch. 2-3]{Fokianos2002a} for other early references.
Recently, binary time series data have been increasingly popular in various financial applications
(\cite{Breen1989},
\cite{Butler1992},
\cite{Christoffersen2006},
\cite{Christoffersen2007},
\cite{startz2008binomial},
\cite{Nyberg2010a,nyberg2011forecasting,nyberg2013predicting},
\cite{Kauppi2012}
and
\cite{wu2014parameter}),
but also to other scientific fields.
Previous results related to  theoretical properties of binary time series models were given by \cite{deJong2011}.

Related work on categorical time series has been reported by \cite{Fahrmeir1987b}, \citet{Kaufmann(1987)}, \citet{Fokianos2003}
and \citet{Engle1998,Engle2005} who proposed a categorical time series model  for  financial transactions data.
Alternative classes of models are  based on the probit link function.
Such autoregressive models have been considered by \citet{Zeger1988a}, \citet{Rydberg2003},
\citet{Kauppi2008a},
among others. Several alternative classes of  models for the analysis of categorical data have been studied; see the books by \citet{Joe(1997)}
and \citet{MacDonald1997}
and the  articles by \citet{Biswas2009} and
\citet{Weiss2011}.

To prove the theoretical results, we will be  assuming a contraction type condition; such conditions are usually employed for the theoretical analysis of time
series models.  For instance, in the   case of count time series models,  see \cite{Fok}, \cite{Neumann2011} and
\cite{Doukhan2012c}.  However, our work is closely  related to the modeling approach suggested by
\citet{Fokianos2011b}, because the main idea is, essentially, to employ the so called
canonical link process to model the observed data.
Note that \cite{deJong2011} have shown  near epoch dependence for a binary time series models but these authors have a different modeling point of view.

Likelihood based inference for the models we study  can be developed along the lines of previous references by appealing to the
properties of  multinomial
distribution.  The proof of consistency and asymptotic normality for the maximum likelihood estimator
is based on standard arguments concerning  convergence of the score function and the Hessian matrix. In addition, differentiability
properties of multinomial likelihood allow obtaining suitable bounds for higher order derivatives.
However, we mention that this work relaxes considerably previous results.
For the case of a model with covariates we improve considerably upon \citet{Kaufmann(1987)} and \citet[Ch.3]{Fokianos2002a} because we avoid any assumptions on the design of covariates. In addition, we  show that we  obtain ergodicity even in the case of finite order models with
unbounded covariates. The study of maximum
likelihood estimation requires existence of appropriate moments for the covariate process though.
Central limit theorems for the maximum likelihood estimators have been given in the previous references
and therefore we do not give any further details.

Section \ref{sec:simple AR} discusses  general categorical time series  models
by allowing the conditional probabilities to depend on the whole past of the series. In addition we will be giving  a result about the stationarity and ergodicity of chains with complete connections.
These results   will be applied to the case of an infinite order autoregressive  multinomial logistic
model. Section \ref{sec:Latent AR} discuss models which might include a latent process. The results obtained by
Theorem \ref{main} improve the results obtained by  \citet{Fok1} and \citet{Fok2}. Finally, Section \ref{sec:cov}
discuss inclusion of exogenous covariates to the autoregressive multinomial logistic model of finite order.
Theorem \ref{goodresult}, which  is the main result in this section, proves  existence of such processes and their ergodicity.

\section{Time series autoregressive  models for categorical data}
\label{sec:simple AR}

\subsection{A general approach}
Let ${\cal A}$ be a finite set. For simplicity, we assume that ${\cal A}=\{1,2,\ldots,N\}$, where $N$ is a nonnegative integer. Suppose that we observe a process
with state space ${\cal A}$ and we are interested on modeling its dynamics. For instance, consider modeling of a stock price change (0 for no change, 1 for positive change and -1 for a negative change; see \citet{Engle1998}),
sleep state status (see \citet{Fokianos2003}), and  wage mobility data (see \citet{PammingerandSchnatter(2010)}).
Towards this goal,  define a $(N-1)$-dimensional vector $Y_t=\left(Y_{1t},Y_{2t},\ldots,Y_{(N-1)t}\right)'$, for $1\leq t\leq n$, such that
\begin{eqnarray*}
Y_{kt} = \left\{ \begin{array}{ll}
1, & \textrm{if the $k$'th category is observed at time $t$,}\\
0, & \textrm{otherwise},
\end{array} \right.
\end{eqnarray*}
for all $k=1,2,\ldots, N-1$. Throughout this work,  consider a stochastic processes $(Y_t)_{t\in \Z}$ adapted to  a  filtration $\left(\mathcal{F}_t\right)_{t\in\Z}$ which is defined through a vector of conditional "success"
probabilities, say $p_{t}\equiv(p_{1t},p_{2t},\ldots,p_{(N-1)t})'$. In other words
\begin{equation}\label{def}
p_{kt}=\P\left(Y_{kt}=1\vert \mathcal{F}_{t-1}\right),\quad 1\leq j\leq N-1.
\end{equation}
For the  last category $N$,  set $Y_{Nt}=1-\sum_{k=1}^{N-1}Y_{kt}$ and
$p_{Nt}=1-\sum_{k=1}^{N-1}p_{kt}$.

There are several possibilities for autoregressive modeling of processes that take values on a finite space.
For instance, assuming  that  $d$ is  a vector and $A$, $B$ matrices of appropriate dimension,
consider the following linear model
\begin{equation}
\label{TM:S:DMBCTS:E:the special multi model for lag=1 with identity link,noc=m}
    p_{t} = d + A p_{t-1} + B Y_{t-1},\quad t\in \mathbb{Z},
\end{equation}
which was studied  by  \citet{Engle1998} and \citet{Qaqish2003}. Model
\eqref{TM:S:DMBCTS:E:the special multi model for lag=1 with identity link,noc=m}
implies quite complex  restrictions on the parameters $d, A$ and $B$  because each element of the vector
$p_t$ has to  belong in the interval $(0,1)$.  Such restrictions become even more involved when a covariate process is included in \eqref{TM:S:DMBCTS:E:the special multi model for lag=1 with identity link,noc=m}. To avoid such subtle technicalities, we adapt the generalized linear models point of view by
considering a canonical link model;
see \citet{Fokianos2003} for instance.
For $k=1,2,\ldots,N-1$, define
$$\lambda_{kt}=\log\left(p_{kt}/p_{Nt}\right)$$
and suppose that the  vector process  $\lambda_t=\left(\lambda_{1t},\ldots,\lambda_{(N-1)t}\right)'$ is determined by the
infinite order model
\begin{equation}\label{cool1}
\lambda_t=g\left(Y_{t-1},Y_{t-2},\ldots\right),
\end{equation}
where $g$ is a suitably defined function. Then, the
process $(Y_t)_{t\in\Z}$ which satisfies  (\ref{def}) and (\ref{cool1}),
takes its values in the set $E=\{e_1,e_2,\ldots,e_{N-1},{\bf 0}\}$
where $\{e_1,\ldots,e_{N-1}\}$ is the canonical basis of $\R^{N-1}$ and ${\bf 0}$ is the null vector of $\R^{N-1}$.
Furthermore, $g:E^{\N}\rightarrow \R^{N-1}$ is a measurable function and
the conditional distribution of $Y_t$ given  its past values $Y_{t-1}^{-} \equiv \left(Y_{t-1},Y_{t-2},\ldots\right)$
possibly depends on its infinite past. A useful  example of such process is given by the linear process
\begin{equation}
\label{eq:linear no feedback}
\lambda_{t} = d + \sum_{j \geq 1} A_{j} Y_{t-j}
\end{equation}
where $d$ is a $(N-1)$-dimensional vector and $(A_{j})_{j \geq 1}$ is a sequence of $(N-1)\times (N-1)$ matrices.
Comparison of  \eqref{eq:linear no feedback} to  \eqref{TM:S:DMBCTS:E:the special multi model for lag=1 with identity link,noc=m}
shows that unnecessary  restrictions on the unknown coefficients can be circumvented  since the vector $\lambda_{t} \in \R^{N-1}$.
Furthermore, covariates can be easily included in \eqref{eq:linear no feedback} by including an additional additive term.
Other categorical type autoregressive  models can be considered but \eqref{eq:linear no feedback} has been  used in several applications.
In the case that $N=2$, then \eqref{eq:linear no feedback} is a simple logistic regression model which has been studied widely in the literature (see \citet{CoxandSnell(1970)}
for an early reference).

Autoregressive models, as those we consider in this work,  are  particular
examples of a more general class of processes  which are called \emph{chains with complete connections}.
Such processes  have been  widely studied in applied probability;
\citet{Doeblin}, \citet{Harris} and \citet{Ios}.
Following the work of  \citet{Bres}, we discuss next a  coupling technique related to   chains with complete connections.

\subsection{Some results about chains with complete connection}
Throughout  this section, consider a finite state space $E$.
For $x,y\in E^{\N}$ and a positive integer $m$, we  write $x\stackrel{m}= y$ if $x_i=y_i$ for $0\leq i\leq m-1$.
Consider a probability kernel $p(\cdot \vert \cdot)$ defined on  $\left(E^{\N},\mathcal{B}\left(E^{\N}\right)\right)$
and taking values on  $\left(E,\mathcal{B}(E)\right)$ which satisfies the following assumption:
\begin{description}
\item
[Assumption (A)] There exists a sequence $(\gamma_m)_{m\in\N}$  which decreases  to zero, as $m \rightarrow \infty$, with $\gamma_0<1$ and
such that for $a \in E$
$$\inf_{x,y:x\stackrel{m}= y }\frac{p\left(a\vert x\right)}{p\left(a\vert y\right)}\geq 1-\gamma_m.$$
\end{description}
\emph{A chain with complete connections} is a stationary process satisfying  Assumption \textbf{(A)}.

For $x \in E^{\N}$, consider the chain $\left(Z_n^x\right)_{n\in\Z}$ which satisfies that $Z_{-j}^x=x_j$ for $j\geq 1$ and
$$\P\left(Z_n^x=a\vert Z^x_{n-j}=z_j, j\geq 1\right)=p(a\vert z)\prod_{j=n+1}^{\infty}\mathds{1}_{z_j=x_{j-n}},\quad n\geq 0.$$
In addition, given a real-valued sequence $(\gamma_{n})_{n \in \N}$,
let  the Markov chain $\left(S_n^{(\gamma)}\right)_{n\in\N}$  taking values in $\N$
and  defined by
$$\P\left(S^{(\gamma)}_0=0\right)=1,\quad \P\left(S_{n+1}^{(\gamma)}=i+1\vert S_n^{(\gamma)}=i\right)=1-\gamma_i,\quad \P\left(S_{n+1}^{(\gamma)}=0\vert S_n^{(\gamma)}=i\right)=\gamma_i.$$
For $n \geq 1$, define the quantity
$$
\gamma_n^{*}=\P\left(S_n^{(\gamma)}=0\right),
$$
which plays  a  crucial  rule for evaluating the mixing coefficients of the chain.
The following result is given by  \citet[Prop. 1 and Lemma 1]{Bres}.
\begin{prop}\label{central}
For all $x,y\in E^{\N}$, there is a coupling $\left(\left(U_n^{x,y},V_n^{x,y}\right)\right)_{n\in\Z}$ of $\left(Z_n^x\right)_{n\in\Z}$ and $\left(Z_n^y\right)_{n\in\Z}$ such that the integer-valued process $\left(T_n^{x,y}\right)_{n\in\Z}$ defined by
$$T_n^{x,y}=\inf\left\{m\geq 0: U_{n-m}^{x,y}\neq V_{n-m}^{x,y}\right\},$$
satisfies
$$\P\left(S_n^{(\gamma)}\geq k\right)\leq \P\left(T_n^{x,y}\geq k\right)~~\forall k \in \N.$$
\end{prop}
Proposition \ref{central} is proved by defining
iteratively  the pair $\left(U_n^{x,y},V_n^{x,y}\right)$ using the
maximal coupling of the conditional distributions
$p\left(\cdot\vert (u_{n-j})_{j\geq 1}\right)$ and
$p\left(\cdot\vert (v_{n-j})_{j\geq 1}\right)$ (i.e the coupling associated to the total variation distance between these conditional distributions).
Proposition \ref{central} yields  the following corollary (see also \citet[Cor. 1]{Bres} for a specific case of the following result).

\begin{cor}\label{central2}
For all $k\geq 1$, $x,y\in E^{\N}$ and $B\in \mathcal{B}(E^k)$, we have
$$\left\vert \P\left((Z_n^x,\ldots,Z_{n+k}^x)\in B\right)-\P\left((Z_n^y,\ldots,Z_{n+k}^y)\in B\right)\right\vert
\leq \sum_{j=0}^k \left(\prod_{m=0}^{j-1}\left(1-\gamma_m\right)\right)\gamma^{*}_{n+k-j}.$$
\end{cor}

\begin{proof}%{Proof of Corollary \ref{central}}
Using  Proposition \ref{central}, we obtain
\begin{eqnarray*}
\left\vert \P\left((Z_n^x,\ldots,Z_{n+k}^x)\in B\right)-\P\left((Z_n^y,\ldots,Z_{n+k}^y)\in B\right)\right\vert
&\leq& \P\left((U_n^{x,y},\ldots,U_{n+k}^{x,y})\neq (V_n^{x,y},\ldots,V_{n+k}^{x,y})\right)\\
&\leq & \P\left(T_n^{x,y}\leq k\right).
\end{eqnarray*}
Proposition \ref{central} implies that  $\P\left(T_n^{x,y}\leq k\right)\leq \P\left(S_n^{(\gamma)}\leq k\right)$.
The result of the corollary now follows  by bounding $\P\left(S_n^{(\gamma)}\leq k\right)$ along the lines of the  derivation of
\citet[eq. (4.25)]{Bres}.
\end{proof}

As pointed out in \citet{Bres}, if $\lim_{n\rightarrow \infty}\gamma_n^{*}=0$, then Corollary \ref{central2} implies  existence and uniqueness
of a stationary chain $(Z_n)_{n\in\Z}$ with complete connections and satisfying Assumption \textbf{(A)}.
Furthermore, Corollary \ref{central2} yields a bound for controlling  the $\phi-$mixing coefficients associated with
$\left(Z_n\right)_{n\in\Z}$.
Indeed,
recall  that for two $\sigma-$algebras $\mathcal{A}$ and $\mathcal{B}$, their $\phi-$mixing coefficients are  defined by (see \citet{Doukhan(1994)},
for instance)
$$
\phi\left(\mathcal{A},\mathcal{B}\right)=\sup_{(A,B)\in \mathcal{A}\times\mathcal{B}: \P(A)>0}\Bigl(
\vert \P\left(B\vert A\right)-\P(B)\vert\Bigr).
$$
Then, the  $\phi$-mixing coefficients of the random sequence $\left(Z_n\right)_{n\in\Z}$ are given by
$$
\phi(n)=\phi\left(\mathcal{F}_{-\infty,0},\mathcal{F}_{n,\infty}\right)=\sup_{k\in\N}\phi\left(\mathcal{F}_{-\infty,0},\mathcal{F}_{n,n+k}\right),
$$
because a Borel set on the infinite product can be approximated by a finite union of cylinder sets.
%we have also
%$$\phi(n)=\sup_{k\in\N}\phi\left(\mathcal{F}_{-\infty,0},\mathcal{F}_{n,n+k}\right).$$
\begin{prop}\label{cons}
Suppose  that $\sum_{n\geq 1}\gamma^{*}_n<\infty$. Then the infinite order stationary  Markov chain
$(Z_n)_{n\in\Z}$, which exists by Corollary  \ref{central2},  is $\phi-$mixing with mixing coefficients satisfying
$\phi(n)\leq \sum_{j\geq n}\gamma^{*}_j$.
\end{prop}

\begin{proof}
Suppose that  $\mu$ denotes the probability distribution of $(Z_i)_{i\leq -1}$. Then
\begin{eqnarray*}
\left\vert \P\left((Z_n^x,\ldots,Z_{n+k}^x)\in B\right)-\P\left((Z_n,\ldots,Z_{n+k})\in B\right)\right\vert
& \leq &   \int \left\vert \P\left((Z_n^x,\ldots,Z_{n+k}^x)\in B\right)-\P\left((Z_n^y,\ldots,Z_{n+k}^y)\in B\right)\right\vert \mu(dy)\\
& \leq &    \sum_{j=0}^k \left(\prod_{m=0}^{j-1}\left(1-\gamma_m\right)\right)\gamma^{*}_{n+k-j}.
\end{eqnarray*}
But the last bound does not depend on $x$. Hence, we obtain
\begin{eqnarray*}
\left\vert \P\left((Z_n,\ldots,Z_{n+k})\in B\vert \mathcal{F}_{-\infty,0}\right)-\P\left((Z_n,\ldots,Z_{n+k})\in B\right)\right\vert
&\leq&
\sum_{j=0}^k \left(\prod_{m=0}^{j-1}\left(1-\gamma_m\right)\right)\gamma^{*}_{n+k-j}\\
&\leq& \sum_{j\geq n}\gamma^{*}_j.
\end{eqnarray*}
The last bound, which does not depend on $k$ and $B$, is also an upper bound for $\phi(n)$.
\end{proof}

\begin{Remark} \rm
\label{decay}
It has been shown in  \citet[Prop. 2]{Bres}  that
$$\sum_{k\geq 1}\gamma_k<\infty\Rightarrow \sum_{k\geq 1}\gamma^{*}_k<\infty.$$
Moreover, if $(\gamma_m)_{m}$ decreases exponentially, then so does $\left(\gamma_n^{*}\right)_n$. Hence, the result of
Prop. \ref{cons} follows again and if $(\gamma_m)_{m}$ decreases to zero exponentially fast then so does $(\phi(n))_{n}$.
Note also that the $\phi-$mixing property implies  ergodicity of the process; see \citet[pp. 50-51]{Brad}.
\end{Remark}

\subsection{Application to categorical time series}

Recall the categorical time series model   $(Y_{t})_{t \in \Z}$  whose state space is $E=\left\{e_1,\ldots,e_{N-1},{\bf 0}\right\}$ and
defined by \eqref{def} and \eqref{cool1}.
From the results of the previous subsection, we deduce the following corollary.
($\Vert\cdot\Vert$ denotes  the Euclidian norm on $\R^{N-1}$.)

\begin{cor}\label{appli1}
Assume model \eqref{cool1} and let a function
$g:E^{\N}\rightarrow \R^{N-1}$ be
such that there exist a sequence $(\delta_j)_{j\in\N}$ which satisfies
$\sum_{j\in\N}\sum_{k \geq j} \delta_k<\infty$ and
\begin{equation}\label{lip}
\Vert g(x)-g(y)\Vert \leq \sum_{j\in\N} \delta_j \mathds{1}_{x_j\neq y_j}.
\end{equation}
Then, there exists a unique stochastic process $(Y_t)_{t\in \Z}$ taking values in $E$ such that
\begin{equation}\label{fifi}
\P\left(Y_t= e_j \vert \mathcal{F}_{t-1}\right)=
\frac{\exp\left(g_j\left(Y_{t-1},Y_{t-2},\ldots\right)\right)}{1+\sum_{s=1}^{N-1}\exp\left(g_s\left(Y_{t-1},Y_{t-2},\ldots\right)\right)},~~ 1\leq j\leq N-1,
\end{equation}
where $g_{j}(\cdot)$ is the $j$'th component of $g$.
Moreover $(Y_t)_{t\in\Z}$ is stationary and $\phi-$mixing.
\end{cor}
\begin{proof}

Denote by   $p(\cdot \vert \cdot)$  the probability kernel defined by
$$p(e_j \vert x)=F_j\left[g\left(x_0,x_1,\ldots\right)\right],~~1\leq j \leq N-1,$$
where $F_j:\R^{N-1}\rightarrow [0,1]$ is defined  for $z \in \R^{N-1}$ by
$$F_j(z)=\frac{\exp(z_j)}{1+\sum_{s=1}^{N-1}\exp(z_s)},\quad 1\leq j\leq N-1.$$
Because of \eqref{def}, $F_N(z)=\left(1+\sum_{s=1}^{N-1}\exp(z_s)\right)^{-1}$.
The Lipschitz assumption \eqref{lip} implies that the $j$'th component of $g$  is
bounded, for $j=1,2\ldots,N-1$. Hence, there  exists $\eta>0$ such that, for all $1\leq j\leq N-1$ and $x\in E^{\N}$,
$$\eta\leq p(e_j\vert x),\quad \eta\leq p({\bf 0} \vert x).$$
Moreover, $F'_j$ is bounded, for all $j$.  Set  $M=\max_{1\leq j\leq N-1}\sup_{z\in\R^{N-1}}\Vert F'_j(z)\Vert$.
Then, if $x\stackrel{m}{=}y$ and $a\in E$, we have that
$$\frac{p(a\vert x)}{p(a\vert y)}\geq 1-\frac{M\sum_{j\geq m+1}\delta_j}{\eta}.$$

Provided that  $m$ is  large enough,  choose $\gamma_m=M\sum_{j\geq m+1}\delta_j/\eta$. Hence, there exists  an  $m$ such that  $\gamma_m=1-\eta$.
Then we have $\sum_{k\geq 1}\gamma_k<\infty$ and using Remark \ref{decay}, we have also $\sum_{k\geq 1}\gamma_k^{*}<\infty$.
Then from Corollary \ref{central2} and Proposition \ref{cons}, there exists a unique stationary solution $(Y_t)_{t\in\Z}$ satisfying (\ref{fifi}) and the solution is $\phi-$mixing.
\end{proof}

We note that  the condition $\sum_{j\in \N}\sum_{k\geq j}\delta_k<\infty$ is equivalent to the condition
$\sum_{j\in\N}j\delta_j<\infty$. For the \emph{infinite order}  linear model \eqref{eq:linear no feedback}, Corollary \ref{appli1} applies
provided that $\sum_{j\geq 1}j\Vert A_j\Vert<\infty$ where $\Vert A_j\Vert$ denotes the corresponding operator norm of the matrix $A_j$.
In particular, when $N=2$, we obtain that the logistic autoregressive model of infinite order is stationary and $\phi$-mixing if
$\sum_{j\geq 1}j |A_j| <\infty$, where $(A_{j})_{j \geq 1}$ denotes a real valued sequence.
For the general non linear model defined in (\ref{cool1}) but of finite-order, i.e. $g$ only depends on finitely many $y_j'$s, note that the assumptions of Corollary \ref{appli1} are automatically satisfied because the number of coefficients $\delta_j$ is also finite.

\section{Categorical time series with a latent process}
\label{sec:Latent AR}

In this section, we consider some  specific instances  of chains with complete connections. Following the methodology of GARCH models
(see \citet{Engle(1982)}, \citet{Bollerslev(1986)} and the text by \citet{FrancqandZakoian(2010)} for instance), and recalling the notation introduced in \eqref{cool1}
we model  the latent process $(\lambda_t)_{t\in\Z}$ to depend additionally on its past values.
From a  statistical perspective, such parametrization  yields parsimony and allows for more flexible structures that can accommodate
various forms of autocorrelation.
To be specific, suppose that  $p$ and $q$  are two positive integers and Let $f:\R^{(N-1)p}\times E^q\rightarrow \R^{N-1}$ be a function
such that
\begin{equation}\label{cool}
\lambda_t=f\left(\lambda_{t-1},\ldots,\lambda_{t-p},Y_{t-1},\ldots,Y_{t-q}\right),\quad t\in\Z.
\end{equation}
We  will say that the process $\left((Y_t,\lambda_t)\right)_{t\in\Z}$ is a solution of the problem $\mathcal{P}_f$
if (\ref{cool}) is satisfied and for each $t\in\Z$, $\lambda_t$ is $\mathcal{F}_{t-1}-$measurable.

\subsection{A general result}

For $\underline{y}\in E^q$, define the mapping   $G_y:\R^{(N-1)p}\rightarrow \R^{(N-1)p}$
by
$$G_{\underline{y}}(\underline{x})=\left(f(x_1,\ldots,x_p,y_1,\ldots,y_q)',x_1',\ldots,x_{p-1}'\right)',$$
where $f(\cdot)$ has been defined by \eqref{cool}.  The main result of this section is the following.

\begin{theo}\label{main}
Suppose that there exist an integer $k\geq 1$, $\kappa\in (0,1)$ and $K>0$ such that
$$\Vert G_{\underline{y}}(\underline{x})-G_{\underline{y}'}(\underline{x}')\Vert\leq K\left(\Vert \underline{x}-\underline{x}'\Vert+\mathds{1}_{\underline{y}\neq \underline{y}'}\right),$$
and for all $\underline{x},\underline{x}',\underline{y}_1,\ldots,\underline{y}_k$,
\begin{equation}\label{contrac}
\Vert G_{\underline{y}_1}\circ\cdots\circ G_{\underline{y}_k}(\underline{x})-G_{\underline{y}_1}\circ\cdots\circ G_{\underline{y}_k}(\underline{x}')\Vert\leq \kappa \Vert \underline{x}-\underline{x}'\Vert.
\end{equation}
Then, the following hold true:
\begin{enumerate}
\item
Let $\underline{x}$ be a vector of $\R^{(N-1)p}$
and $\left(\underline{y}_j\right)_{j\geq 1}$ a sequence of elements of $E^q$.
Then the limit
$$\lim_{s\rightarrow \infty}G_{\underline{y}_1}\circ\cdots\circ G_{\underline{y}_s}(x)$$
exists and
does not depend on $x$. Let $H:(E^q)^{\N}\rightarrow \R^{(N-1)p}$ be the function defined by
$$H\left(\underline{y}_1,\underline{y}_2,\ldots\right)=\lim_{s\rightarrow \infty}G_{\underline{y}_1}\circ\cdots\circ G_{\underline{y}_s}(x).$$
Then the function $H$ is bounded. Moreover there exist $C>0$ such that
$$\Vert H\left(\underline{y}_1,\underline{y}_2,\ldots\right)-H\left(\underline{y}'_1,\underline{y}'_2,\ldots\right)\Vert\leq C\sum_{j\geq 1}\kappa^{j/k}\mathds{1}_{\underline{y}_j\neq \underline{y}_j'},$$
where $\kappa$ is defined by (\ref{contrac}).

\item
A process $\left((Y_t,\lambda_t)\right)_{t\in\Z}$ is solution of the problem $\mathcal{P}_f$ is and only if $(Y_t)_{t\in\Z}$ is a chain with complete connection associated to a function $g$ (see Corollary \ref{appli1}) defined by
$$g\left(Y_{t-1},Y_{t-2},\ldots\right)=H_1\left(V_t,V_{t-1},\ldots\right),\quad V_t=(Y_{t-1},\ldots,Y_{t-q}).$$
Here $H_1$ denotes the $N-1$ first coordinates  of the function $H$ defined previously.
\item
There exists a unique strictly stationary solution to the equations (\ref{def}) and (\ref{cool}). Moreover the process
$\left(Y_t\right)_{t\in \Z}$ is $\phi-$mixing with geometrically  decreasing   mixing coefficients. This implies  the ergodicity of
the joint process $\left(\left(Y_t,\lambda_t\right)\right)_{t\in\Z}$.
\end{enumerate}
\end{theo}

\begin{proof}

\begin{enumerate}
\item The first part of the assertion is a straightforward consequence of the assumption  and is omitted. We focus on
the proof of the Lipschitz property of the function $H$. For $j\geq 1$, we set
$$G^{(j)}_{\underline{y}}=G_{\underline{y}_{(j-1)k+1}}\circ G_{\underline{y}_{(j-1)k+2}}\circ\cdots\circ G_{\underline{y}_{jk}}.$$
We have
$$H\left(\underline{y}_1,\underline{y}_2,\ldots\right)=\lim_{s\rightarrow\infty}G^{(1)}_{\underline{y}}\circ\cdots\circ G^{(s)}_{\underline{y}}(x).$$
By the stated assumption, we obtain that
$$\Vert G^{(j)}_{\underline{y}}(x)-G^{(j)}_{\underline{y}'}(x)\Vert\leq \sum_{\ell=1}^kK^{\ell}\mathds{1}_{\underline{y}_{(j-1)k+\ell}\neq \underline{y}'_{(j-1)k+\ell}}.$$
Hence
$$\Vert G^{(1)}_{\underline{y}}\circ\cdots\circ G^{(s)}_{\underline{y}}(x)-G^{(1)}_{\underline{y}'}\circ\cdots\circ G^{(s)}_{\underline{y}'}(x)\Vert\leq \sum_{j=1}^{\infty}\kappa^{j}\sum_{\ell=1}^kK^{\ell}\mathds{1}_{\underline{y}_{(j-1)k+\ell}\neq \underline{y}'_{(j-1)k+\ell}}.$$
By setting $C=\left(K\vee 1\right)^k$ and letting $x\rightarrow\infty$ we obtain the result.

\item
From the first point of the theorem, the necessary condition follows easily. Now let us assume that
$\lambda_t=H_1\left(V_t,V_{t-1},\ldots\right)$. Setting $\underline{\lambda}_t=H\left(V_t,V_{t-1},\ldots\right)$,
note that the continuity of the function $G_{V_t}$ implies  that
\begin{eqnarray*}
\underline{\lambda}_t&=&\lim_{s\rightarrow \infty}G_{V_t}\circ G_{V_{t-1}}\circ\cdots G_{V_{t-s}}(x)\\
&=& G_{V_t}\left(\lim_{s\rightarrow \infty}G_{V_{t-1}}\circ\cdots\circ G_{V_{t-s}}(x)\right)\\
&=& G_{V_t}\left(\underline{\lambda}_t\right),
\end{eqnarray*}
which proofs  that $(\lambda_t)_{t\in\Z}$ satisfies (\ref{cool}).
\item
The third point is a straightforward consequence of the two first results  and of Corollary \ref{appli1}.
Moreover the geometric decay of the $\phi-$mixing coefficients has been discussed in the remarks made  following Proposition \ref{cons}.
Finally, it is well-known that $\phi-$mixing implies  ergodicity of the process $(Y_t)_{t\in\Z}$ and then ergodicity of the process
$\left((\lambda_t,Y_t)\right)_{t\in\Z}$; see \citet[Ch. 2]{Samo}, for instance.
\end{enumerate}

\end{proof}

\subsection{Linear models}
Let $A_0,A_1,\ldots,A_p,B_1,\ldots,B_q$ be some real matrices of size $(N-1)\times (N-1)$. We assume
that
$$f\left(x_1,\ldots,x_p;y_1,\ldots,y_q\right)=A_{0}+ \sum_{i=1}^p A_ix_i+\sum_{i=1}^q B_i y_i.$$
Then the above model can be written alternatively as
$$G_{V_t}\left(\underline{x}\right)=\widetilde{A}
\underline{x}+B,$$
with
$$\widetilde{A}=\left(\begin{array}{cc}\begin{array}{ccc}A_1&\cdots&A_{p-1}\end{array}& A_p\\ I_{(N-1)(p-1)} & 0\end{array}\right),\quad B=\begin{pmatrix} A_0+\sum_{i=1}^q B_i Y_{t-i}\\0\\\vdots\\0\end{pmatrix},$$
where $I_{(N-1)(p-1)}$ denotes the identity matrix of order $(N-1)(p-1)$. Then, the assumptions of Theorem \ref{main} are satisfied if the spectral radius of $\widetilde{A}$ is less than unity (and then the norm of $\widetilde{A}^k$ is less than one if $k$ is large enough) which also implies  that
the roots of the polynomial $\mathcal{P}(z)=\det\left(I_{N-1}-\sum_{i=1}^p A_iz^i\right)$ are outside the unit disc; \citet[Ch.2]{Lut}.
For the case $p=q=1$, this result improves the conditions proved by \citet{Fok1}
since it does not require any additional assumption for the coefficient $B_{1}$. In fact, we
reconfirm and generalize the condition obtained  by \citet{Tjostheim2012} for the case of binary logistic autoregressive model when  $p=q=1$.
Compared with the work of \citet{Fok2} we note again that for
the case of logistic autoregressive modeling with binary data, these  conditions simplify  their findings.

\subsection{Non-linear models}

We discuss now the case of non-linear models. For example,
a non-linear model that can be used for binary time series  is a threshold type model. For a binary time series,
we  can define with some abuse of notation,
the log-odds ratio by
$$\lambda_t=d+\beta_1 \lambda_{t-1}^{+}+\beta_2\lambda_{t-1}^{-}+\alpha Y_{t-1},$$
where $x^{+}=max(x,0)$, $x^{-}=min(x,0)$ and $\max\left\{\vert\beta_1\vert,\vert\beta_2\vert\right\}<1$.
Such a model implies  higher/lower probability of the occurrence of  state $1$, depending on the sign of the lagged  value of the latent process.
Other models that can be introduced along the previous lines for binary time series are the so called
smooth autoregressive models as advocated by \citet{Terasvirta(1994)}.
Such models will allow for smooth transitions between states.

Recall \eqref{cool} and assume  that there exists a norm $\Vert\cdot\Vert$ on $\R^{N-1}$
and some positive real numbers $\beta_1,\ldots,\beta_p,\alpha_1,\ldots,\alpha_q$ with
$\alpha=\sum_{i=1}^p\alpha_i<1$ and for all $x_1,x_1',\ldots,x_p,x'_p,y_1,y_1',\ldots,y_q,y'_q\in \R^{N-1}$,
$$\Vert f\left(x_1,\ldots,x_p;y_1,\ldots,y_q\right)-f\left(x'_1,\ldots,x'_p;y'_1,\ldots,y'_q\right)\Vert\leq \sum_{i=1}^p\alpha_i\Vert x_i-x'_i\Vert+\sum_{i=1}^q \beta_i\Vert y_i-y'_i\Vert.$$
It can be proved under this  condition that,  for a process $(W_t)_{t\in\Z}$ defined by $W_{t}=(Y^{T}_{t-1}, \ldots, Y^{T}_{t-q})^{T}$and
 taking values in $\R^{(N-1)q}$, the random mapping
$$G_{W_t}(x)=\left(f(x,W_t),x_1,\ldots,x_{p-1}\right)$$
is contracting, after iteration. Indeed, if $x, x' \in\R^{p}$, $\lambda^x_i=x_i$ for $1\leq i\leq p$ and
$$\lambda_t^x=f\left(\lambda^{x}_{t-1},\ldots,\lambda^{x}_{t-p},W_t\right),\quad t\geq p+1,$$
it follows by induction that
$$\left\Vert \lambda_t^x-\lambda_t^{x'}\right\Vert\leq \alpha^{\frac{t-p}{p}}\Vert x-x'\Vert,\quad t\geq p+1.$$
Hence, there exists an integer $m\geq 1$, such that the mapping
$$H^{(m)}_t(x)=G_{W_t}\circ G_{W_{t-1}}\circ\cdots\circ G_{W_{t-m}}(x)$$
satisfies
$$\left\Vert H^{(m)}_t(x)-H^{(m)}_t(x')\right\Vert\leq \kappa \Vert x-x'\Vert$$
for some $\kappa\in (0,1)$.
Therefore the assumption of Theorem \ref{main} is satisfied. We note again that this condition improves upon
the conditions obtained by \citet{Fok1} and \citet{Fok2} since they require only that $\alpha < 1$.

\section{Inclusion of exogenous covariates for finite-order Markov chain models}
\label{sec:cov}

In this section, we study the problem of including a covariate process $(Z_t)_{t\in\Z}$
in an autoregressive categorical time series model. We will be  assuming
that the covariate process  is  strongly exogenous and unbounded.
We focus on  the case of finite order Markov chains, i.e. the parameter $\lambda_t$ does not depend on its past values.
The general case seems more difficult to tackle and will not be studied in the present paper; see Remark \ref{remarkimp}
for more.

We are not aware of any result which guarantees
ergodicity of a model with covariate even in the simple case of a  finite-state Markov chain. As we will see, there is an interesting parallel
between Markov chains with exogenous covariates and Markov chains in random environments which have been  studied in probability theory.
In the proof of Theorem \ref{goodresult} given below, we will use an approach discussed in \citet{Cog} for showing ergodicity of Markov processes in random environments.

\subsection{A general result for finite state Markov chains with covariates}

We will be discussing  results concerning stationarity and ergodicity for a special case of time inhomogeneous Markov
chain. We will consider  a finite state Markov chain which can be jointly observed with a  covariate process.
In what follows, denote by  $Z=\left(Z_t\right)_{t\in\Z}$ a stationary process with values in a measurable space $\left(G,\mathcal{G}\right)$ and
$(Y_t)_{t\in \N}$ a process which takes  values in a finite set $E$. In addition, conditionally on $Z$, $(Y_t)_{t\in \N}$ is
a finite-state inhomogeneous Markov chain. More precisely, we assume that there exist a family of transition matrices $\left\{P_g:g\in G\right\}$ such that
\begin{equation}\label{cond}
\P\left(Y_t=y\vert Y_{t-1}=x;Z\right)=P_{Z_t}(x,y),\quad (x,y)\in E^2.
\end{equation}
Throughout the section we will assume the following:

\begin{description}
\item[(E1)]
There exists an integer $m\geq 1$ such that for all $(z_1,z_2,\ldots,z_m)\in G^m$, the product of stochastic matrices $P_{z_1}P_{z_2}\cdots P_{z_m}$ has positive entries.

\item[(E2)]
The process $Z$ is mixing in the ergodic theory sense, i.e for all elements $A$ and $B$ of $\mathcal{B}\left(G^{\Z}\right)$, we have
$$\lim_{n\rightarrow \infty}\P\left(Z\in A, \tau^n Z\in B\right)=\P(Z\in A)\P(Z\in B),$$
where $\tau$ denotes the shift operator on $F^{\Z}$ defined by $\tau Z=\left(Z_{j+1}\right)_{j\in \Z}$.
\end{description}

\begin{Remark} \rm

Note that  Assumption ({\bf E1}) implies  that a process $(Y_t)_{t\in\Z}$ satisfying (\ref{cond}) also satisfies
$$\P\left(Y_{t+m}=y\vert Y_t=x,Z\right)>0,\mbox{ a.s }\quad (x,y,t)\in E\times E\times \Z.$$
In addition,  Assumption ({\bf E2}) is  stronger than  assuming  ergodicity of the process $Z$ but weaker than
the classical strong mixing condition usually employed in the literature. A large number  of useful stochastic processes are mixing, for instance the strong mixing processes and
Bernoulli shifts defined by  $Z_t=H\left(\varepsilon_t,\varepsilon_{t-1},\ldots\right)$ where $H$ is a measurable function and $\left(\varepsilon_t\right)_{t\in \Z}$ is an i.i.d sequence.
\citet[Ch.2]{Samo}, discusses several  properties of the different types of mixing in ergodic theory of stationary processes.
Assumption ({\bf E2}) will be employed  for obtaining ergodicity for the shift operator $\tau^m$ which is not implied by the ergodicity of the shift operator $\tau$.
\end{Remark}

\begin{Remark} \rm
Model  (\ref{cond}) implies an  exogeneity assumption for the covariate process $(Z_t)_t$. In particular,
for each time $t$, $Y_t$ is independent from $(Z_s)_{s\geq t+1}$ conditionally to $Y_{t-1},Z_t,Y_{t-2},Z_{t-1}\ldots$
which allows for  simple computation of the conditional likelihood function. Indeed, we have for $y_1,\ldots,y_n\in E$,
$$\P\left(Y_2=y_2,\ldots,Y_n=y_n\vert Y_1=y_1,Z_1,\ldots, Z_n\right)=\prod_{i=1}^nP_{Z_i}(y_{i-1},y_i).$$
This type of exogeneity is also called Granger causality or Sims causality in the literature (see, for instance, \citet[Sec. 1.5.2]{Gour}, for a discussion of these different concepts).
\end{Remark}

\begin{Remark} \rm
We  can assume more generally that
$$\P\left(Y_t=y \vert Z, Y_{t-1}=x\right)=P_{\widetilde{Z}_t}(x,y),\quad \widetilde{Z}_t=\left(Z_{t-j}\right)_{j\geq 0}.$$
But this case is  already covered by assumptions (\textbf{E1}) and (\textbf{E2}).
Indeed, if the process $(Z_t)_{t}$ is assumed to be stationary and mixing, then the process $(\widetilde{Z}_t)_{t}$ taking values in $G^{\Z}$ is also stationary and mixing by using  \citet[Prop. 2.2.4 and Cor. 2.2.5]{Samo}. Hence,
if we assume ({\bf E1}) for the covariate process $(\widetilde{Z}_t)_{t}$ taking values in $G^{\Z}$,  assumption ({\bf E2}) is automatically
satisfied for the covariate process $(\widetilde{Z}_t)_{t}$ provided that $(Z_t)_{t}$ is mixing.
In other words, there is  no loss of generality by assuming that the time-inhomogeneous Markov chain only depends on the coordinate $Z_t$ of the covariate process. This point is important in applications because past values of covariates are routinely included in a regression model.

\end{Remark}

The main result of this section is given in the following theorem.

\begin{theo}\label{goodresult}
Suppose that  ({\bf E1-E2}) hold true. Then  there exists a unique stochastic process $(Y_t)_{t\in\Z}$ satisfying (\ref{cond}).
Moreover the process $\left((Y_t,Z_t)\right)_{t\in\Z}$ is ergodic.
\end{theo}

\begin{Remark}\rm
Developing theory for models (\ref{cool1}) and (\ref{cool}) which  include exogenous covariates
is a more challenging problem, as we explain next.
Indeed, the proof of ergodicity relies heavily  on assumption ({\bf E1})
which asserts the positivity of the conditional density of the model we study.
For simplicity, assume that $p=q=1$ and $N=2$ in (\ref{cool}) with
$$\lambda_t=f\left(\lambda_{t-1},Y_{t-1},Z_t\right).$$
Then, the bivariate process $\left((Y_t,\lambda_t)\right)_t$ is  (conditionally on $Z$)
a time-inhomogeneous Markov chain with transition kernel
$$Q_{Z_t}\left((y,\lambda),\{1\}\times A\right)=F\left(f(\lambda,y,Z_t)\right)\delta_{f(\lambda,y,Z_t)}(A),$$
where for $(a,z)\in \R^2$,
$F(a)={\exp(a)}/({1+\exp(a)})$ and $\delta_z$ denotes the Dirac mass at point $z$.
Therefore, the  transition kernel $Q_{Z_t}$ is  not absolutely continuous with respect
to a measure which depends on $(\lambda ,y)$. When there is no covariate,
the Markov chain does not satisfy an irreducibility assumption and we recover
a classical problem occurring to  observation-driven time series models such as the Poisson GARCH type models;
see for instance \citet{Fok},  \citet{Neumann2011} and \citet{Douc}.
In Section \ref{sec:Latent AR}, we avoided this problem and proved ergodicity using a general coupling result for processes
which are not necessarily Markov.
If the dynamics of the model  are specified conditionally to the realization of a covariate process, then
a non-homogeneous version of this coupling technique might be possible to be derived
for defining  the conditional distribution of $Y\vert Z$ and then to prove existence of a stationary stochastic process $\left((Y_t,Z_t)\right)_{t\in \Z}$ which satisfies  the aforementioned recursions, by following  the first part of the proof of Theorem \ref{goodresult}.
However,  the probability kernel $Q$  does not have not a positive density w.r.t  a dominating measure and therefore
the second part of this  proof cannot be adapted when covariates are included.
In general,
studying ergodicity properties of a stochastic process which includes simultaneously a latent process and a covariate process is  more challenging problem and will be considered elsewhere.
\label{remarkimp}
\end{Remark}

\begin{proof}
We first show that the almost sure limit $\lim_{s\rightarrow \infty}P_{Z_{t-s}}\cdots P_{Z_t}(x,y)$ exists for each $y\in E$ and does not depend on $x$.
For Markov chains, this condition is comparable to the weak ergodicity notion, but here the limit is taken in the backward sense.
See \citet{Seneta} for several sufficient conditions ensuring weak ergodicity properties of time-inhomogeneous Markov chains, using ergodicity coefficients.
Recall  that the so-called Dobrushin's contraction coefficient of a stochastic matrix $P$ is defined by
$$c(P)=\frac{1}{2}\sup_{x\neq y\in E}\Vert P(x,\cdot)-P(y,\cdot)\Vert_{TV},$$
where  for two probability measures $\mu$ and $\nu$ on the finite set $E$, the total variation distance between $\mu$ and $\nu$ is defined by $\Vert\mu-\nu\Vert_{TV}=\sum_{x\in E}\left\vert \mu(x)-\nu(x)\right\vert$.
It is well known that we have the contraction
$$\Vert \mu P-\nu P\Vert_{TV}\leq c(P)\Vert\mu-\nu\Vert_{TV}$$
and for two stochastic matrices $P$ and $Q$, we have $c(PQ)\leq c(P)c(Q)$.

\noindent
Moreover $c(P)\leq 1-\vert E\vert \min_{x,y\in E}P(x,y)$, where $\vert E\vert$ denotes the cardinality of the set $E$.
So Assumption {\bf (E1)} ensures that for all $t\in \Z$, $c\left(P_{Z_t}P_{Z_{t+1}}\cdots P_{Z_{t+m-1}}\right)<1$ a.s.
Now let $x\neq x'\in E$, $t\in\Z$ and $s=km+\ell$.  we obtain by  setting $\rho=1-\eta\vert E\vert$,
$$\Vert P_{Z_{t-s+1}}\cdots P_{Z_t}(x,\cdot)-P_{Z_{t-s+1}}\cdots P_{Z_t}(x',\cdot)\Vert_{TV}\leq 2c\left(P_{Z_{t-km+1}}\cdots P_{Z_t}\right)\leq 2\prod_{j=0}^{k-1} c\left(P_{Z_{t-(j+1)m+1}}\cdots P_{Z_{t-jm}}\right).$$
From Assumption {\bf (E2)}, the covariate process $Z$ is mixing. Then the process $\left(Z_{t-j}\right)_{j\in\Z}$ is also mixing. Indeed, if $\theta_t$ and $\tau$ denote the mappings defined on $G^{\Z}$ by $\theta_t x=(x_{t-i})_{i\in\Z}$ and $\tau x=(x_{i+1})_{i\in\Z}$ respectively, we have $\tau\circ \theta_t=\theta_t\circ \tau^{-1}$. Then for two Borel sets  $A$and $B$, we get
\begin{eqnarray*}
\P\left(\theta_t Z\in A, \tau^n\theta_t Z\in B\right)&=& \P\left(Z\in \theta_t^{-1} A, \tau^{-n} Z\in \theta_t^{-1}B\right)\\
&=& \P\left(\tau^n Z\in \theta_t^{-1} A, Z\in \theta_t^{-1}B\right)\\
&\rightarrow& \P\left(\theta_t Z\in A\right)\P\left(\theta_t Z\in B\right).
\end{eqnarray*}
Moreover, observe that the operator $\tau^m$ is ergodic for $\P_{Z}$.
Indeed, if a Borel  set $A$ is such that $\tau^m A=A$, we obtain, using assumption {\bf E2},
$$\P\left(Z\in A\right)=\P\left(\tau^{km}Z\in A,Z\in A\right)\rightarrow_{k\rightarrow \infty}\P\left(Z\in A\right)^2.$$
Then, we conclude that $\P_Z(A)\in \{0,1\}$, which shows that $\tau^m$ is ergodic.
Now, using Assumption {\bf (E1)}, we have $\E \log c\left( P_{Z_1}\cdots P_{Z_m}\right)<0$.
Then, from the ergodic theorem, we get
$$\prod_{j=0}^{k-1} c\left(P_{Z_{t-(j+1)m+1}}\cdots P_{Z_{t-jm}}\right)=\exp\left(\sum_{j=0}^{k-1}\log c\left(P_{Z_{t-(j+1)m+1}}\cdots P_{Z_{t-jm}}\right)\right)\rightarrow_{k\rightarrow \infty}0.$$

\noindent
In addition,  when  $n\geq s$, we deduce that
$$\Vert P_{Z_{t-s+1}}\cdots P_{Z_t}(x,\cdot)-P_{Z_{t-n}}\cdots P_{Z_t}(x,\cdot)\Vert_{TV}\leq 2c\left(P_{Z_{t-s+1}}\cdots P_{Z_t}\right).$$
From the Cauchy criterion, we deduce that the product of matrices $P_{Z_{t-s+1}}\cdots P_{Z_t}$ converges, when $s\rightarrow \infty$, to
a stochastic matrix whose rows  are all  equal. Then there exists a measurable function $D:G^{\N}\rightarrow E^N$ with $N=\vert E\vert$ such that
$$D\left(Z_t,Z_{t-1},\ldots\right)=\lim_{s\rightarrow \infty} P_{Z_{t-s+1}}\cdots P_{Z_t}(x,\cdot)\mbox{ a.s}.$$
Setting $\overline{D}_t=D\left(Z_t,Z_{t-1},\ldots\right)$, $\overline{D}_t$ is a random probability measure on $E$.
For $t\in\Z$, $z\in G^{\Z}$, $k$ a non-negative integer and $y_0,y_1,\ldots,y_k\in E$, we set
$$\mu_{t:t+k}(z;y_0,y_1,\ldots,y_k)=\prod_{i=1}^kP_{z_{t+i}}(y_{i-1},y_i)\overline{D}_t\left(y_0\right).$$
From the Kolmogorov extension theorem, there exists for $\P_{Z}-$almost all values of $z\in \mathcal{B}\left(G^{\Z}\right)$ a unique measure $\mu(z,\cdot)$ on $E^{\Z}$ with marginals $\mu_{t:t+k}(z,\cdot)$. Hence,
if $\zeta$ denotes the probability distribution of $Z$,
the measure $\gamma$ defined by
$$\gamma(A\times B)=\int_B \mu(z,A)\zeta(dz),\quad (A,B)\in \mathcal{B}(E^{\Z})\times \mathcal{B}(F^{\Z}),$$
is that of a couple from a  stationary process $(Y,Z)$ satisfying (\ref{cond}).
\noindent
To show uniqueness,  let $\left(Y'_t\right)_{t\in\Z}$ be another stochastic process satisfying (\ref{cond}). Then the distribution of $Y'\vert Z=z$ is that of a non-homogeneous Markov chain with transitions $\{P_{z_t}:t\in\Z\}$. As shown before, this conditional distribution is unique and equal to $\mu(z,\cdot)$.

\noindent
Next we show ergodicity of the process $\left((Y_t,Z_t)\right)_{t\in\Z}$. To this end, we use an approach introduced in \citet{Cog} for the study of Markov processes in random environment. This type of argument is also used in \citet{Sinn} for positive transition matrix $P_{Z_t}$ and we give here a more general and shorter proof. The approach used in \citet{Cog} consists in considering the Markov kernel $Q$ on $E\times G^{\Z}$ defined by
$$Q\left((x,z),\{y\}\times A\right)=P_{z_1}(x,y)\mathds{1}_A(\tau z),\quad A\in\mathcal{B}\left(F^{\Z}\right).$$
If $\nu$ denotes the probability distribution $(Y_t,\tau^t Z)$ which takes values in $E\times G^{\Z}$, then $\nu$ is invariant for $Q$ and
the process $(H_t)_{t\in\Z}$ defined by $H_t=(Y_t,\tau^t Z)$ is a Markov chain of transition kernel $Q$.
Let $C$ be a $\nu-$ invariant set, i.e $Q((x,z),C)=1$ for $\nu-$almost every $(x,z)\in C$.
Using Corollary $5.11$ in \citet{Hairer}, the Markov chain $(H_t)_{t\in\Z}$ forms an ergodic process if and only if every $\nu-$invariant set $C$ is of measure $0$ or $1$.
In our case, we have $C=\cup_{x\in E}\{x\}\times C_x$ for some $C_x\in\mathcal{B}\left(F^{\Z}\right)$.
To this end, we first note  that if $C$ is $\nu-$invariant, then
$$\nu(C)=\nu Q(C)=\int_C d\nu(x,z)Q((x,z),C)+\int_{C^c}d\nu(x,z)Q((x,z),C)=\nu(C)+\int_{C^c}d\nu(x,z)Q((x,z),C).$$
Then we get $Q((x,z),C)=0$ for $\nu-$almost every $(x,z)\in C^c$, the complement of $C$ in $E\times G^{\Z}$.
Hence, we obtain  $Q((x,z),C)=Q\mathds{1}_C(x,z)=\mathds{1}_C(x,z)$ for $\nu-$almost every $(x,z)$.
But this also gives $Q^m\mathds{1}_C=\mathds{1}_C$, $\nu$ a.e, where $m$ has been defined by  assumption ({\bf E1}).
Moreover,  we have that
\begin{equation}\label{egal}
Q^m((x,z),C)=\sum_{y\in E}\mathds{1}_{C}(y,\tau^mz)\left[P_{z_1}\cdots P_{z_m}\right](x,y).
\end{equation}
We write $A=B$ $\nu-$a.e. if $\nu\left(A\Delta B\right)=0$ where $A\Delta B$ denotes the symmetric difference of the sets $A$ and $B$, i.e $A\Delta B=(A\cap B^c)\cup(A^c\cap B)$.
From assumption ({\bf E1}), all the entries of the matrix $P_{z_1}\cdots P_{z_m}$ are positive. Then we deduce that for
almost every $(x,z)\in C$, we have $(y,\tau^m z)\in C$ for all $y\in E$. We set $D=\cap_{y\in E}C_y$.
Let us denote by $\nu_1$ and $\nu_2$ the marginals of $\nu$. We first note that for all $A\in\mathcal{B}\left(G^{\Z}\right)$, we have
$$\nu\left(\{x\}\times A\right)=\int_A\P\left(Y_0=x\vert Z=z\right)\nu_2(dz)=\sum_{y\in E}\int_A\P\left(Y_0=x\vert Y_{-m}=y,Z=z\right)\P\left(Y_n=y\vert Z=z\right)\nu_2(dz).$$
We get
$$\nu\left(\{x\}\times A\right)\geq \int_A \eta_z \nu_2(dz),\quad \eta_z=\min_{x,y\in E}\P\left(Y_0=x\vert Y_{-m}=y,Z=z\right).$$
Employing again assumption ({\bf E1}), we have $\eta_z>0$ for all $z$ and we deduce that $\nu_2(A)=0$ as soon as $\nu\left(\{x\}\times A\right)=0$.
Now for $x\in E$, we set $B_x=\tau^m C_x\setminus D$. As stated above we have
$$\nu\left(\{(x,z): z\in B_x\}\right)=\sum_{x\in E}\nu\left(\{x\}\times B_x\right)=0.$$
We conclude that $\nu_2(B_x)=0$ for all $x\in E$ and then $\nu_2\left(\tau^mC_x\setminus C_x\right)=0$. By stationarity,
$\nu_2\left(\tau^m C_x\right)=\nu_2\left(C_x\right)$. Therefore,  for every $x\in E$, $\tau^m C_x=C_x$, $\mu-$a.e.
But using assumption {\bf E2}, we have that
$$\nu_2\left(C_x\right)=\nu_2\left(\tau^{km} C_x\cap C_x\right)\rightarrow_{k\rightarrow \infty}\nu_2\left(C_x\right)^2.$$
Then, we conclude that $\nu_2(C_x)\in \{0,1\}$. If $\nu_2(C_x)=0$ for every $x$, we easily get $\nu(C)=0$. Now if there exists
$x\in E$ such that $\nu_2(C_x)=1$, we have, using the equality $\nu_2(B_x)=0$,
$$1\leq \nu_2(C_x)=\nu_2\left(\tau^m C_x\right)=\nu_2\left(\tau^m C_x\cap D\right)\leq \nu_2(D)\leq \min_{y\in E}\nu_2(C_y).$$
Then $\nu_2(C_y)=1$ for each $y\in E$. Finally we obtain
$$\nu(C)=\sum_{y\in E}\nu\left(\{y\}\times C_y\right)=\sum_{y\in E}\nu_1(y)=1.$$
Hence, we have shown that the process $\left(H_t\right)_{t\in \Z}$ is ergodic and so is the process $\left((Y_t,Z_t\right)_{t\in\Z}$.

\end{proof}

\subsection{Application to the multinomial logistic model with covariates}

We assume here that conditionally to a covariate process $(Z_t)_{t\in\Z}$ taking values in $\R^d$, the process $(Y_t)_{t\in\Z}$ is a $q-$order Markov chain such that
$$\P\left(Y_t=e_j\vert Y_{t-1},\ldots,Y_{t-q},Z\right)=\frac{\exp\left(g_j\left(Y_{t-1},\ldots,Y_{t-q};Z_t\right)\right)}{1+\sum_{s=1}^{N-1}\exp\left(g_s\left(Y_{t-1},\ldots,Y_{t-q};Z_t\right)\right)}:=Q_{Z_t}\left(Y_{t-q:t-1},e_j\right)$$
for some measurable functions $g_j:E^q\times \R^d$, $1\leq j\leq N-1$.
Let us check that assumption {\bf E1} is satisfied for the conditional Markov chain $(X_t)_{t\in\Z}$ defined
by $X_t=\left(Y_t^{T},Y_{t-1}^{T},\ldots,Y_{t-q+1}^{T}\right)^{T}$. Conditionally to $Z$, the process $\left(X_t\right)_{t\in\Z}$ defines a time-inhomogeneous Markov chains such that
\begin{eqnarray*}
P_{Z_t}\left((u_1,\ldots,u_q),(v_1,\ldots,v_q)\right):&=& Q_{Z_t}\left((u_1,\ldots,u_q),v_1\right)\prod_{s=1}^{q-1}\mathds{1}_{v_{s+1}=u_s}.
\end{eqnarray*}
Since the transition $Q_{Z_t}$ takes only positive values, the assumption {\bf E1} follows by taking
$m=q$. Then assuming {\bf E2} for the covariate process, Theorem \ref{goodresult} applies and guarantees the ergodicity of the process $\left((Y_t,Z_t)\right)_{t\in\Z}$. These results show that our approach simplifies conditions required to obtain consistency and asymptotic
normality of the maximum likelihood estimator, even in the case of considering covariates. However, existence of moments for the covariate process
is still required to study large sample properties the maximum likelihood estimator.

{Comparing these results with the work of  \citet{Kaufmann(1987)}, we note that likelihood inference can be developed without
making assumptions on the design of covariates and assuming that they are bounded. The work by \citet{Kaufmann(1987)} did not study ergodicity
properties of categorical time series models with covariates though.
Additional previous work by \citet{Fokianos2002a} on likelihood estimation
employs an assumption regarding ergodicity of the joint process $(Y^{T}_{t}, Z^{T}_{t})^{T}$ (see Assumption A in \citet[pp.16--17]{Fokianos2002a}).
Theorem \ref{goodresult} shows that such assumptions are not necessary,  at least in the context of categorical time
series models.  Finally, comparing our work with that of \citet{deJong2011} who
consider the case of a probit model for  binary time series we see that the results we obtain apply in this  case
without assuming an alpha-mixing condition   on the covariate process.

\section*{Acknowledgements}
Part of this work was done while K. Fokianos was with the Department of Mathematics \& Statistics, University of Cyprus. The authors
thank the Editor, the Associate Editor and a referee for several useful comments and suggestions.


\newpage

\begin{thebibliography}{}

\bibitem[\protect\citeauthoryear{Biswas and Song}{Biswas and
  Song}{2009}]{Biswas2009}
Biswas, A. and P.~X.-K. Song (2009).
\newblock Discrete-valued {ARMA} processes.
\newblock {\em Statistics \& Probability Letters\/}~{\em 79}, 1884--1889.

\bibitem[\protect\citeauthoryear{Bollerslev}{Bollerslev}{1986}]{Bollerslev(1986)}
Bollerslev, T. (1986).
\newblock Generalized autoregressive conditional heteroskedasticity.
\newblock {\em Journal of Econometrics\/}~{\em 31}, 307--327.

\bibitem[\protect\citeauthoryear{Bradley}{Bradley}{2007}]{Brad}
Bradley, R. (2007).
\newblock {\em Introduction to Strong Mixing Conditions. {V}ol. 1}.
\newblock Kendrick Press, Heber City, UT.

\bibitem[\protect\citeauthoryear{Breen, Glosten, and Jagannathan}{Breen
  et~al.}{1989}]{Breen1989}
Breen, W., L.~R. Glosten, and R.~Jagannathan (1989).
\newblock Economic significance of predictable variations in stock index
  returns.
\newblock {\em The Journal of Finance\/}~{\em 44}, 1177--1189.

\bibitem[\protect\citeauthoryear{Bressaud, Fern\'andez, and Galves}{Bressaud
  et~al.}{1999}]{Bres}
Bressaud, X., R.~Fern\'andez, and A.~Galves (1999).
\newblock Decay of correlations for non-{H}\"{o}lderian dynamics. a coupling
  approach.
\newblock {\em Electron. J. Probab.\/}~{\em 4}, 1--19.

\bibitem[\protect\citeauthoryear{Butler and Malaikah}{Butler and
  Malaikah}{1992}]{Butler1992}
Butler, K.~C. and S.~Malaikah (1992).
\newblock Efficiency and inefficiency in thinly traded stock markets: Kuwait
  and {S}audi {A}rabia.
\newblock {\em Journal of Banking \& Finance\/}~{\em 16}, 197--210.

\bibitem[\protect\citeauthoryear{Christoffersen, Diebold, Mariano, Tay, and
  Tse}{Christoffersen et~al.}{2007}]{Christoffersen2007}
Christoffersen, P., F.~X. Diebold, R.~S. Mariano, A.~Tay, and Y.~K. Tse (2007).
\newblock Direction-of-change forecasts for asian equity markets based on
  conditional variance, skewness and kurtosis dynamics: Evidence from hong kong
  and singapore.
\newblock {\em Journal of Financial Forecasting\/}~{\em 1}, 1--22.

\bibitem[\protect\citeauthoryear{Christoffersen and Diebold}{Christoffersen and
  Diebold}{2006}]{Christoffersen2006}
Christoffersen, P.~F. and F.~X. Diebold (2006).
\newblock Financial asset returns, direction-of-change forecasting, and
  volatility dynamics.
\newblock {\em Management Science\/}~{\em 52}, 1273--1287.

\bibitem[\protect\citeauthoryear{Cogburn}{Cogburn}{1984}]{Cog}
Cogburn, R. (1984).
\newblock The ergodic theory of {M}arkov chains in randon environments.
\newblock {\em Z. Wahrscheinlichkeitstheory verw. Gebiete\/}~{\em 66},
  109--128.

\bibitem[\protect\citeauthoryear{Cox}{Cox}{1981}]{Cox1981}
Cox, D.~R. (1981).
\newblock Statistical analysis of time series: some recent developments.
\newblock {\em Scand. J. Statist.\/}~{\em 8}, 93--115.

\bibitem[\protect\citeauthoryear{Cox and Snell}{Cox and
  Snell}{1970}]{CoxandSnell(1970)}
Cox, D.~R. and E.~J. Snell (1970).
\newblock {\em The Analysis of Binary Data}.
\newblock London: Chapman \& Hall.

\bibitem[\protect\citeauthoryear{de~Jong and Woutersen}{de~Jong and
  Woutersen}{2011}]{deJong2011}
de~Jong, R.~M. and T.~Woutersen (2011).
\newblock Dynamic time series binary choice.
\newblock {\em Econometric Theory\/}~{\em 27}, 673--702.

\bibitem[\protect\citeauthoryear{Doeblin and Fortet}{Doeblin and
  Fortet}{1937}]{Doeblin}
Doeblin, W. and R.~Fortet (1937).
\newblock Sur les chaξnes ΰ liaisons complθtes.
\newblock {\em Bull. Soc. Math. France\/}~{\em 65}, 132--148.

\bibitem[\protect\citeauthoryear{Douc, Doukhan, and Moulines}{Douc
  et~al.}{2013}]{Douc}
Douc, R., P.~Doukhan, and E.~Moulines (2013).
\newblock Ergodicity of observation-driven time series models and consistency
  of the maximum-likelihood estimator.
\newblock {\em Stochastic Processes and their Applications\/}~{\em 123},
  2620--2647.

\bibitem[\protect\citeauthoryear{Doukhan}{Doukhan}{1994}]{Doukhan(1994)}
Doukhan, P. (1994).
\newblock {\em Mixing: properties and examples}.
\newblock Number~85 in Lecture Notes in Statistics. New York: Springer-Verlag.

\bibitem[\protect\citeauthoryear{Doukhan, Fokianos, and Tj{\o}stheim}{Doukhan
  et~al.}{2012}]{Doukhan2012c}
Doukhan, P., K.~Fokianos, and D.~Tj{\o}stheim (2012).
\newblock On weak dependence conditions for {P}oisson autoregressions.
\newblock {\em Statist. Probab. Lett.\/}~{\em 82}, 942--948.

\bibitem[\protect\citeauthoryear{Engle}{Engle}{1982}]{Engle(1982)}
Engle, R.~F. (1982).
\newblock Autoregressive conditional heteroscedasticity with estimates of the
  variance of {U}nited {K}ingdom inflation.
\newblock {\em Econometrica\/}~{\em 50}, 987--1007.

\bibitem[\protect\citeauthoryear{Fahrmeir and Kaufmann}{Fahrmeir and
  Kaufmann}{1987}]{Fahrmeir1987b}
Fahrmeir, L. and H.~Kaufmann (1987).
\newblock Regression models for nonstationary categorical time series.
\newblock {\em Journal of Time Series Analysis\/}~{\em 8}, 147--160.

\bibitem[\protect\citeauthoryear{Fahrmeir and Tutz}{Fahrmeir and
  Tutz}{2001}]{Fahrmeir2001}
Fahrmeir, L. and G.~Tutz (2001).
\newblock {\em Multivariate statistical modelling based on generalized linear
  models\/} (Second ed.).
\newblock Springer Series in Statistics. New York: Springer-Verlag.
\newblock With contributions by Wolfgang Hennevogl.

\bibitem[\protect\citeauthoryear{Fokianos and Kedem}{Fokianos and
  Kedem}{2003}]{Fokianos2003}
Fokianos, K. and B.~Kedem (2003).
\newblock Regression theory for categorical time series.
\newblock {\em Statist. Sci.\/}~{\em 18}, 357--376.

\bibitem[\protect\citeauthoryear{Fokianos and Moysiadis}{Fokianos and
  Moysiadis}{2017}]{Fok2}
Fokianos, K. and T.~Moysiadis (2017).
\newblock Binary time series friven by a latent process.
\newblock {\em Econometrics and Statistics\/}~{\em 2}, 117--130.

\bibitem[\protect\citeauthoryear{Fokianos, Rahbek, and Tjostheim}{Fokianos
  et~al.}{2009}]{Fok}
Fokianos, K., A.~Rahbek, and D.~Tjostheim (2009).
\newblock Poisson autoregression.
\newblock {\em J. Amer. Statist. Assoc.\/}~{\em 104}, 1430--1439.

\bibitem[\protect\citeauthoryear{Fokianos and Tj{\o}stheim}{Fokianos and
  Tj{\o}stheim}{2011}]{Fokianos2011b}
Fokianos, K. and D.~Tj{\o}stheim (2011).
\newblock Log-linear {P}oisson autoregression.
\newblock {\em J. Multivariate Anal.\/}~{\em 102}, 563--578.

\bibitem[\protect\citeauthoryear{Francq and Zako{\"{\i}}an}{Francq and
  Zako{\"{\i}}an}{2010}]{FrancqandZakoian(2010)}
Francq, C. and J.-M. Zako{\"{\i}}an (2010).
\newblock {\em {GARCH} models: Stracture, Statistical Inference and Financial
  Applications}.
\newblock United Kingdom: Wiley.

\bibitem[\protect\citeauthoryear{Gourieroux and Monfort}{Gourieroux and
  Monfort}{1995}]{Gour}
Gourieroux, C. and A.~Monfort (1995).
\newblock {\em Statistics and Econometric Models. Volume 1}.
\newblock Cambridge University Press.

\bibitem[\protect\citeauthoryear{Hairer}{Hairer}{2006}]{Hairer}
Hairer, M. (2006).
\newblock {\em Ergodic properties of Markov processes}.
\newblock Lecture notes available at
  \url{http://www.hairer.org/notes/Markov.pdf}.

\bibitem[\protect\citeauthoryear{Harris}{Harris}{1955}]{Harris}
Harris, T. (1955).
\newblock On chains of infinite order.
\newblock {\em Pacific J. Math.\/}~{\em 5}, 707--724.

\bibitem[\protect\citeauthoryear{Iosifescu and Grigorescu}{Iosifescu and
  Grigorescu}{1990}]{Ios}
Iosifescu, M. and S.~Grigorescu (1990).
\newblock {\em Dependence with Complete Connections and its Applications}.
\newblock Cambridge University Press.

\bibitem[\protect\citeauthoryear{Joe}{Joe}{1997}]{Joe(1997)}
Joe, H. (1997).
\newblock {\em Multivariate Models and Dependence Concepts}.
\newblock London: Chapman \& Hall.

\bibitem[\protect\citeauthoryear{Kaufmann}{Kaufmann}{1987}]{Kaufmann(1987)}
Kaufmann, H. (1987).
\newblock Regression models for nonstationary categorical time series:
  Asymptotic estimation theory.
\newblock {\em Annals of Statistics\/}~{\em 15}, 79--98.

\bibitem[\protect\citeauthoryear{Kauppi}{Kauppi}{2012}]{Kauppi2012}
Kauppi, H. (2012).
\newblock Predicting the direction of the {F}ed's target rate.
\newblock {\em Journal of Forecasting\/}~{\em 31}, 47--67.

\bibitem[\protect\citeauthoryear{Kauppi and Saikkonen}{Kauppi and
  Saikkonen}{2008}]{Kauppi2008a}
Kauppi, H. and P.~Saikkonen (2008).
\newblock Predicting {US} recessions with dynamic binary response models.
\newblock {\em The Review of Economics and Statistics\/}~{\em 90}, 777--791.

\bibitem[\protect\citeauthoryear{Kedem}{Kedem}{1980}]{Kedem1980book}
Kedem, B. (1980).
\newblock {\em Binary Time Series}.
\newblock Marcel Dekker, New York.

\bibitem[\protect\citeauthoryear{Kedem and Fokianos}{Kedem and
  Fokianos}{2002}]{Fokianos2002a}
Kedem, B. and K.~Fokianos (2002).
\newblock {\em Regression models for time series analysis}.
\newblock Hoboken, NJ: Wiley.

\bibitem[\protect\citeauthoryear{L{\"u}tkepohl}{L{\"u}tkepohl}{2005}]{Lut}
L{\"u}tkepohl, H. (2005).
\newblock {\em New Introduction to Multiple Time Series Analysis\/} (1st ed.).
\newblock Berlin: Springer.

\bibitem[\protect\citeauthoryear{MacDonald and Zucchini}{MacDonald and
  Zucchini}{1997}]{MacDonald1997}
MacDonald, I.~L. and W.~Zucchini (1997).
\newblock {\em Hidden {M}arkov and Other Models for Discrete--valued Time
  Series}.
\newblock London: Chapman \& Hall.

\bibitem[\protect\citeauthoryear{McCullagh and Nelder}{McCullagh and
  Nelder}{1989}]{McCullaghandNelder(1989)}
McCullagh, P. and J.~A. Nelder (1989).
\newblock {\em Generalized Linear Models\/} (2nd ed.).
\newblock London: Chapman \& Hall.

\bibitem[\protect\citeauthoryear{Moysiadis and Fokianos}{Moysiadis and
  Fokianos}{2014}]{Fok1}
Moysiadis, T. and K.~Fokianos (2014).
\newblock On binary and categorical time series models with feedback.
\newblock {\em J. Multivariate Anal.\/}~{\em 131}, 209--228.

\bibitem[\protect\citeauthoryear{Neumann}{Neumann}{2011}]{Neumann2011}
Neumann, M. (2011).
\newblock Absolute regularity and ergodicity of poisson count processes.
\newblock {\em Bernoulli\/}~{\em 17}, 1268--1284.

\bibitem[\protect\citeauthoryear{Nyberg}{Nyberg}{2010}]{Nyberg2010a}
Nyberg, H. (2010).
\newblock Dynamic probit models and financial variables in recession
  forecasting.
\newblock {\em Journal of Forecasting\/}~{\em 29}, 215--230.

\bibitem[\protect\citeauthoryear{Nyberg}{Nyberg}{2011}]{nyberg2011forecasting}
Nyberg, H. (2011).
\newblock Forecasting the direction of the us stock market with dynamic binary
  probit models.
\newblock {\em International Journal of Forecasting\/}~{\em 27}, 561--578.

\bibitem[\protect\citeauthoryear{Nyberg}{Nyberg}{2013}]{nyberg2013predicting}
Nyberg, H. (2013).
\newblock Predicting bear and bull stock markets with dynamic binary time
  series models.
\newblock {\em Journal of Banking \& Finance\/}~{\em 37}, 3351--3363.

\bibitem[\protect\citeauthoryear{Pamminger and Fr\"uhwirth-Schnatter}{Pamminger
  and Fr\"uhwirth-Schnatter}{2010}]{PammingerandSchnatter(2010)}
Pamminger, C. and S.~Fr\"uhwirth-Schnatter (2010).
\newblock Model-based clustering of categorical time series.
\newblock {\em Bayesian Anal.\/}~{\em 5}, 345--368.

\bibitem[\protect\citeauthoryear{Qaqish}{Qaqish}{2003}]{Qaqish2003}
Qaqish, B.~F. (2003).
\newblock A family of multivariate binary distributions for simulating
  correlated binary variables with specified marginal means and correlations.
\newblock {\em Biometrika\/}~{\em 90}, 455--463.

\bibitem[\protect\citeauthoryear{Russell and Engle}{Russell and
  Engle}{1998}]{Engle1998}
Russell, J.~R. and R.~F. Engle (1998).
\newblock {Econometric Analysis of Discrete-Valued Irregularly-Spaced Financial
  Transactions Data Using a New Autoregressive Conditional Multinomial Model}.
\newblock {\em SSRN eLibrary\/}.

\bibitem[\protect\citeauthoryear{Russell and Engle}{Russell and
  Engle}{2005}]{Engle2005}
Russell, J.~R. and R.~F. Engle (2005).
\newblock A discrete-state continuous-time model of financial transactions
  prices and times.
\newblock {\em Journal of Business and Economic Statistics\/}~{\em 23},
  166--180.

\bibitem[\protect\citeauthoryear{Rydberg and Shephard}{Rydberg and
  Shephard}{2003}]{Rydberg2003}
Rydberg, T.~H. and N.~Shephard (2003).
\newblock Dynamics of trade-by-trade price movements: decomposition and models.
\newblock {\em Journal of Financial Econometrics\/}~{\em 1}, 2--25.

\bibitem[\protect\citeauthoryear{Samorodnitsky}{Samorodnitsky}{2016}]{Samo}
Samorodnitsky, G. (2016).
\newblock {\em Stochastic Processes and Long Range Dependence}.
\newblock Springer.

\bibitem[\protect\citeauthoryear{Seneta}{Seneta}{2006}]{Seneta}
Seneta, E. (2006).
\newblock {\em Non-negative matrices and {M}arkov chains}.
\newblock Springer Series in Statistics. Springer, New York.

\bibitem[\protect\citeauthoryear{Sinn and Poupart}{Sinn and
  Poupart}{2011}]{Sinn}
Sinn, M. and P.~Poupart (2011).
\newblock Asymptotic theory for linear-chain conditional random fields.
\newblock In G.~Gordon, D.~Dunson, and M.~Dudik (Eds.), {\em Proceedings of the
  Fourteenth International Conference on Artificial Intelligence and
  Statistics}, Volume~15 of {\em Proceedings of Machine Learning Research},
  Fort Lauderdale, FL, USA, pp.\  679--687.

\bibitem[\protect\citeauthoryear{Slud and Kedem}{Slud and
  Kedem}{1994}]{SludKedem1994}
Slud, E.~V. and B.~Kedem (1994).
\newblock Partial likelihood analysis of logistic regression and
  autoregression.
\newblock {\em Statist. Sinica\/}~{\em 4}, 89--106.

\bibitem[\protect\citeauthoryear{Startz}{Startz}{2008}]{startz2008binomial}
Startz, R. (2008).
\newblock Binomial autoregressive moving average models with an application to
  {U}{S} recessions.
\newblock {\em Journal of Business \& Economic statistics\/}~{\em 26}, 1--8.

\bibitem[\protect\citeauthoryear{Stern and Coe}{Stern and
  Coe}{1984}]{Stern1984}
Stern, R.~D. and R.~Coe (1984).
\newblock A model fitting analysis of daily rainfall data.
\newblock {\em Journal of the Royal Statistical Society. Series A\/}~{\em 147},
  pp. 1--34.

\bibitem[\protect\citeauthoryear{Ter{\"a}svirta}{Ter{\"a}svirta}{1994}]{Terasvirta(1994)}
Ter{\"a}svirta, T. (1994).
\newblock Specification, estimation, and evaluation of smooth transition
  autoregressive models.
\newblock {\em Journal of the American Statistical Association\/}~{\em 89},
  208--218.

\bibitem[\protect\citeauthoryear{Tj{\o}stheim}{Tj{\o}stheim}{2012}]{Tjostheim2012}
Tj{\o}stheim, D. (2012).
\newblock Rejoinder on: {S}ome recent theory for autoregressive count time
  series.
\newblock {\em TEST\/}~{\em 21}, 469--476.

\bibitem[\protect\citeauthoryear{Wei{\ss}}{Wei{\ss}}{2011}]{Weiss2011}
Wei{\ss}, C. (2011).
\newblock Generalized choice models for categorical time series.
\newblock {\em J. Statist. Plann. Inference\/}~{\em 141}, 2849--2862.

\bibitem[\protect\citeauthoryear{Wu and Cui}{Wu and
  Cui}{2014}]{wu2014parameter}
Wu, R. and Y.~Cui (2014).
\newblock A parameter-driven logit regression model for binary time series.
\newblock {\em Journal of Time Series Analysis\/}~{\em 35}, 462--477.

\bibitem[\protect\citeauthoryear{Zeger and Qaqish}{Zeger and
  Qaqish}{1988}]{Zeger1988a}
Zeger, S.~L. and B.~Qaqish (1988).
\newblock Markov regression models for time series: a quasi-likelihood
  approach.
\newblock {\em Biometrics\/}~{\em 44}, 1019--1031.

\end{thebibliography}
\end{document}